\documentclass{article}

\usepackage{amsmath,amssymb,amsfonts,amsthm}
\usepackage{geometry}
\usepackage{graphicx}
\usepackage{algorithm,algorithmic}
\usepackage{url}
\usepackage[rgb]{xcolor}
\usepackage{authblk}

\graphicspath{{fig/}}

\newtheorem{theorem}{Theorem}
\newtheorem{lemma}{Lemma}
\allowdisplaybreaks

\DeclareMathOperator{\diag}{diag}

\DeclareMathOperator{\tridiag}{tridiag}

\newcommand{\mi}{\mathrm{i}}

\newcommand{\trans}{^{\mathsf T}}
\newcommand{\herm}{^{\mathsf H}}
\newcommand{\itrans}{^{-\mathsf T}}

\newcommand{\set}[1]{\left\lbrace#1\right\rbrace}
\newcommand{\conj}[1]{\overline{#1}}

\newcommand{\GT}{\succ}

\renewcommand{\Re}{\mathrm{Re}}
\renewcommand{\Im}{\mathrm{Im}}
\renewcommand{\tilde}{\widetilde}

\newcommand{\blas}{BLAS}
\newcommand{\lapack}{LAPACK}
\newcommand{\pblas}{PBLAS}
\newcommand{\scalapack}{ScaLAPACK}

\newcommand{\hide}[1]{}
\newcommand{\CY}[1]{{\color{blue}~\textsf{[CY: #1]}}}

\title{Structure Preserving Parallel Algorithms for Solving the
Bethe--Salpeter Eigenvalue Problem}
\author[1]{Meiyue Shao}
\author[2,3]{Felipe H. da Jornada}
\author[1]{Chao Yang}
\author[4]{Jack Deslippe}
\author[2,3]{Steven~G.~Louie}
\affil[1]{Computational Research Division,
Lawrence Berkeley National Laboratory, Berkeley, CA 94720}
\affil[2]{Department of Physics,
University of California, Berkeley, CA 94720}
\affil[3]{Materials Sciences Division,
Lawrence Berkeley National Laboratory, Berkeley, CA 94720}
\affil[4]{NERSC, Lawrence Berkeley National Laboratory, Berkeley, CA 94720}

\begin{document}

\maketitle

\begin{abstract}
The Bethe--Salpeter eigenvalue problem is a dense structured eigenvalue
problem arising from discretized Bethe--Salpeter equation in the context of
computing exciton energies and states.
A computational challenge is that at least half of the eigenvalues and the
associated eigenvectors are desired in practice.
We establish the equivalence between Bethe--Salpeter eigenvalue problems and
real Hamiltonian eigenvalue problems.
Based on theoretical analysis, structure preserving algorithms for a class of
Bethe--Salpeter eigenvalue problems are proposed.
We also show that for this class of problems all eigenvalues obtained from the
Tamm--Dancoff approximation are overestimated.
In order to solve large scale problems of practical interest, we discuss
parallel implementations of our algorithms targeting distributed memory
systems.
Several numerical examples are presented to demonstrate the efficiency and
accuracy of our algorithms.

\medskip

\textbf{Keywords:}
Bethe--Salpeter equation,
Tamm--Dankoff approximation,
Hamiltonian eigenvalue problems,
structure preserving algorithms,
parallel algorithms
\end{abstract}

\section{Introduction}
\label{sec:introduction}
The absorption of a photon by a molecular system or solid can promote an
electron in an occupied single-particle state (or orbital) to an unoccupied
state while keeping the charge neutrality.
In the physics community, this process is often described as the simultaneous
creation of a negatively charged quasielectron (or simply electron) and a
positively charged quasihole (or hole) in the material that was originally in
the lowest energy electronic configuration (the ground state).
Upon absorbing a photon, the entire molecular or extended system is in an
excited state that contains a correlated electron--hole pair, which is
referred to as an \emph{exciton}.
The amount of energy required to trigger this excitation gives an important
characterization of the material.
In many-body physics, a two-particle collective excitation is often described
by a two-particle Green's function, with the excitation energy level being a
pole of this function.
It has been shown that the two-particle Green's function satisfies an equation
often known as the \emph{Bethe--Salpeter equation} (BSE)~\cite{SB1951}.

The poles of the two-particle Green's function can be obtained by
computing the eigenvalues of a Hamiltonian operator $\mathcal{H}$ 
associated with this Green's function.  It can be shown that, with an
appropriate discretization scheme, the finite dimensional representation 
of the Bethe--Salpeter Hamiltonian has the following block structure
\begin{equation}
\label{eq:H-complex}
H=\begin{bmatrix} A & B \\ -\conj B & -\conj A \end{bmatrix},
\end{equation}
where $A,B\in \mathbb{C}^{n\times n}$, with
\begin{equation}
\label{eq:symmetry}
A=A\herm, \qquad B=B\trans.
\end{equation}
Here we use $A\herm$ to denote the conjugate transpose of $A$ and
$B\trans$ to denote the transpose of $B$.
We will refer to an eigenvalue problem of the form~\eqref{eq:H-complex} with
the additional symmetry given by Equation~\eqref{eq:symmetry} as a
\emph{Bethe--Salpeter eigenvalue problem}.

In principle, we are interested in all possible excitation energies, although
some excitations are more likely to occur than others.
Such likelihood can often be measured in term of what is known as the
\emph{spectral density} or \emph{density of states} of \(H\), which is defined
to be the number of eigenvalues per unit energy interval~\cite{LSY2014}, that
is,
\begin{equation}
\label{eq:DOS}
\phi(\omega)=\frac{1}{2n}\sum_{j=1}^{2n}\delta(\omega-\lambda_j).
\end{equation}
where $\lambda_j$'s denote the eigenvalues of \(H\).
This formulation requires all eigenvalues of \(H\) to be real, which is the
case for most physical systems.
In addition, the optical absorption spectrum
\begin{equation}
\label{eq:absp}
\epsilon^+(\omega)=\sum_{j=1}^n
\frac{(d_r\herm x_j)(y_j\herm d_l)}{y_j\herm x_j}\delta(\omega-\lambda_j),
\end{equation}
which can be measured in optical absorption experiments, is also of practical
interest.
Here \(d_r\) and \(d_l\) are \emph{dipole vectors}, \(x_j\) and \(y_j\) are
right and left eigenvectors, respectively, corresponding to the
\emph{positive} eigenvalue \(\lambda_j\).
To obtain highly accurate representations of~\eqref{eq:DOS}
and~\eqref{eq:absp}, the computation of all eigenpairs is required.
In addition to the optical absorption spectrum defined in~\eqref{eq:absp}, the
individual pairs of left and right eigenvectors are often desired, since they
describe the character of each two-particle excited state.

\hide{
Although it is possible to estimate the spectral density of $H$
without computing its eigenvalues, such estimate may have a limited
resolution. A more direct and straightforward way to obtain the 
spectral density is to compute all eigenvalues of $H$. \CY{not sure if we should
mention this.}
}

In general, both \(A\) and \(B\) are dense. The dimension of 
these matrices is proportional to the product of the number 
of occupied and unoccupied states, both of which are proportional
to the number of electrons $n_e$ in the system.  Hence it can 
become quite large for large systems that contain many atoms (and 
electrons).

We are interested in efficient and reliable parallel algorithms for computing 
all eigenvalues and the corresponding eigenvectors of \eqref{eq:H-complex}.
Because $H$ is a non-Hermitian matrix, we need to compute both 
the left and the right eigenvectors.

Although it is possible to treat $H$ as a general non-Hermitian matrix and use
existing parallel algorithms~\cite{Fahey2003,GKKS2014} implemented in
ScaLAPACK~\cite{ScaLAPACK} to solve such a problem, this approach does not
take advantage of the special structure of the Bethe--Salpeter Hamiltonian and
is thus not efficient.
Nor does this approach preserve some desirable properties of the eigenvalues
and eigenvectors.
Moreover, the current release of ScaLAPACK only has a subroutine that performs
a Schur decomposition of a complex non-Hermitian matrix~\cite{Fahey2003}.

In the following section, we show that $H$ belongs to a class of
matrices known as \emph{\(J\)-symmetric matrices} whose eigenvalues
satisfy a special symmetry property. Although several algorithms
have been developed for computing the eigenvalues and eigenvectors
of this class of matrices, efficient parallel implementations of
these algorithms are not available, and are not easy to develop.

In this paper, we develop a special algorithm that can leverage existing
parallel computational kernels available in the ScaLAPACK software package.
Our algorithm is based on the observation that computing the eigenpairs
of~\eqref{eq:H-complex} is equivalent to computing the eigenpairs of a real
Hamiltonian matrix of the form
\begin{equation}
\label{eq:H_r}
H_r=\begin{bmatrix} \Im(A+B) & -\Re(A-B) \\ \Re(A+B) & \Im(A-B) \end{bmatrix}
\end{equation}
where $\Re(A)$ and $\Im(A)$ denote the real and imaginary parts
of $A$, respectively.
Furthermore, when $A$ and $B$ satisfy the property
%
\begin{equation}
\label{eq:cond}
\begin{bmatrix} A & B \\ \conj B & \conj A \end{bmatrix}\GT0,%
\footnote{\(X\GT Y\) means that \(X-Y\) is Hermitian positive definite.}
\end{equation}
which often holds in practice~\cite{Zimmermann1970}, it can be shown that all
eigenvalues of $\mi H_r$ and thus of $H$ are real, and they come in positive
and negative pairs~\cite{GMG2009,GMG2011}.
We present an efficient parallel algorithm for computing the positive
eigenvalues and the corresponding right eigenvectors.
The algorithm makes use of existing kernel in ScaLAPACK as well as a new
kernel we developed for computing eigenpairs of a skew symmetric tridiagonal
matrix.
A simple transformation can be used to obtain the right eigenvectors
associated with the negative eigenvalues, as well as all left eigenvectors.

When $H$ is real, which is the case for systems with real-space inversion
symmetry, the Bethe--Salpeter eigenvalue problem can be transformed into a
product eigenvalue problem.
We propose an efficient and accurate parallel algorithm for solving the
product eigenvalue problem.

When facing the challenge of computing the eigenpairs of the non-Hermitian
matrix \eqref{eq:H-complex}, many researchers in the physics community choose
to drop the off-diagonal blocks, $B$ and $-\conj{B}$, and compute eigenpairs
of the Hermitian matrix $A$ only.
This approach is often known as the \emph{Tamm--Dancoff approximation}
(TDA)~\cite{Dancoff1950,RL2000,Tamm1945}.
We show that when the condition \eqref{eq:cond} holds, each eigenvalue of $A$
is an upper bound of the corresponding positive eigenvalue of $H$ when all
eigenvalues of $A$ and $H$ are sorted.
Our numerical experiment shows that TDA can introduce a non-negligible shift
of the spectral density of $H$, which is consistent with what has been
reported in the physics literature~\cite{GMG2009,PMA2013,Zimmermann1970}.


The rest of the paper is organized as follows.
In Section~\ref{sec:properties}, we discuss some basic properties of the
Bethe--Salpeter eigenvalue problem.
Then in Section~\ref{sec:algorithms}, we develop structure preserving
algorithms built on these properties as well as the additional
assumption~\eqref{eq:cond}.
Finally, we demonstrate the efficiency and accuracy of our proposed algorithms
in Section~\ref{sec:experiments} by several examples from physical models.

\section{Properties of Bethe--Salpeter eigenvalue problems}
\label{sec:properties}
In this section, we examine the special properties of the Bethe--Salpeter
Hamiltonian that allow us to develop efficient algorithms for computing its
eigenpairs.

\subsection{Relation to Hamiltonian eigenvalue problems}
\label{subsec:HEP}
We first show that $H$ belongs to a class of matrices known as \(J\)-symmetric
matrices.
Let \(J_n\) be the \(2n\times2n\) skew-symmetric matrix
\begin{equation}
\label{eq:J}
J_n=\begin{bmatrix} 0 & I_n \\ -I_n & 0 \end{bmatrix}.
\end{equation}
A matrix \(X\in\mathbb{C}^{2n\times2n}\) is called a \emph{\(J\)-symmetric
matrix} if \((XJ_n)\trans=XJ_n\)~\cite{MMT2003,MMT2006}.
When \(X\) is real, it is also called a \emph{Hamiltonian} matrix.
By definition, \(X\) is \(J\)-symmetric if and only if \(X\) admits the block
structure
\[
X=\begin{bmatrix} X_{11} & X_{12} \\ X_{21} & -X_{11}\trans \end{bmatrix}
\]
with \(X_{12}\) and \(X_{21}\) complex symmetric.
It can also be verified that this \(J\)-symmetric structure is preserved
under what is called a \emph{complex symplectic similarity transformation}
\[
\hat{H} = S^{-1}HS,
\]
where $S\in\mathbb{C}^{2n\times2n}$ is a \emph{complex symplectic matrix}
that satisfies \(S\trans J_nS=J_n\)~\cite{MMT2006}.
These properties are key to the development of several numerical algorithms
for computing the eigenvalues of dense \(J\)-symmetric matrices.
Examples of these algorithms include the Hamiltonian--Jacobi
algorithms~\cite{Byers1990,Mehl2008} and SR-like algorithms~\cite{Fassbender2007}.

The Bethe--Salpeter Hamiltonian matrix defined in~\eqref{eq:H-complex} is
clearly a \(J\)-symmetric matrix because
\[
H=\begin{bmatrix} A & B \\ -\conj B & -\conj A \end{bmatrix}
=\begin{bmatrix} A & B \\ -\conj B & -A\trans \end{bmatrix}.
\]
Although the algorithms for \(J\)-symmetric matrices can be used to solve the
Bethe--Salpeter eigenvalue problems, they do not take advantage of the
additional symmetry relationship between the \((1,2)\) and \((2,1)\) blocks in
\(H\).
This symmetry leads to the symmetry of \(\Lambda(H)\) as stated in the
following theorem.

\begin{theorem}[\cite{BFY2015}]
\label{thm:complex-symmetry}
Let \(H\) be of the form~\eqref{eq:H-complex} with \(A\) Hermitian and \(B\)
symmetric.
If \(\lambda\) is an eigenvalue of \(H\), then \(-\lambda\), \(\conj\lambda\),
\(-\conj\lambda\) are also eigenvalues of \(H\) with the same multiplicity.
\end{theorem}

Unfortunately, complex symplectic transformations in general do \emph{not}
preserve the structure of Bethe--Salpeter Hamiltonian matrices.
To seek a class of fully structure preserving similarity transformations, we
show in the following theorem that solving a Bethe--Salpeter eigenvalue
problem is equivalent to solving a real Hamiltonian eigenvalue problem.
This is the main theoretical result of this paper.

\begin{theorem}
\label{thm:equiv}
A Bethe--Salpeter eigenvalue problem can be reduced to a real Hamiltonian
eigenvalue problem, and vice versa.
\end{theorem}
\begin{proof}
Let
\[
Q=\frac{1}{\sqrt{2}}
\begin{bmatrix} I_n & -\mi I_n \\ I_n & \mi I_n \end{bmatrix}.
\]
Then
\[
Q\herm\begin{bmatrix} A & B \\ -\conj B & -\conj A \end{bmatrix}Q
=\mi\begin{bmatrix} \Im(A+B) & -\Re(A-B) \\ \Re(A+B) & \Im(A-B) \end{bmatrix}
=-\mi J_n M,
\]
where
\begin{equation}
\label{eq:embedding}
M=\begin{bmatrix} \Re(A+B) & \Im(A-B) \\ -\Im(A+B) & \Re(A-B) \end{bmatrix}
\end{equation}
is a real symmetric matrix.
Therefore, any Bethe--Salpeter eigenvalue problem can be converted to a real
Hamiltonian eigenvalue problem.

On the other hand, let \(H_r\) be a real \(2n\times2n\) Hamiltonian matrix of
the form
\[
H_r=\begin{bmatrix} H_{11} & H_{12} \\ H_{21} & -H_{11}\trans \end{bmatrix},
\]
where \(H_{12}\trans=H_{12}\), \(H_{21}\trans=H_{21}\).
We set
\[
A=\frac{H_{12}-H_{21}}{2}+\mi\frac{H_{11}\trans-H_{11}}{2},
\qquad B=-\frac{H_{12}+H_{21}}{2}-\mi\frac{H_{11}\trans+H_{11}}{2}.
\]
It can be verified that \(A\herm=A\), \(B\trans=B\), and
\[
QH_rQ\herm=\mi\begin{bmatrix} A & B \\ -\conj B & -\conj A \end{bmatrix}.
\]
Therefore, we can convert a real Hamiltonian eigenvalue problem to a
Bethe--Salpeter eigenvalue problem.
This completes the proof.
\end{proof}

Theorem~\ref{thm:equiv} fully characterizes the spectral properties of general
Bethe--Salpeter eigenvalue problems.
Theorem~\ref{thm:complex-symmetry} is a direct consequence of
Theorem~\ref{thm:equiv}.
As a result, several existing algorithms (see~\cite{BKM2005}) for solving 
Hamiltonian eigenvalue problems can be applied to the matrix \(H_r\) defined
in~\eqref{eq:H_r}.

The primary interest of this paper is drawn to the case when the
property~\eqref{eq:cond} holds.
In this case, the Bethe--Salpeter eigenvalue problem is essentially a
symmetric eigenvalue problem because
\begin{equation}
\label{eq:GMG}
H=\begin{bmatrix} I_n & 0 \\ 0 & -I_n \end{bmatrix}
\begin{bmatrix} A & B \\ \conj B & \conj A \end{bmatrix}
\end{equation}
is the product of two Hermitian matrices, and in addition, one of them is
positive definite~\cite{GMG2009}.
Therefore, \(H\) is diagonalizable and has real spectrum.
In addition, the eigenvectors of \(H\) also admit special structures.
These properties are summarized in the following theorem.

\begin{theorem}
\label{thm:positive-definite}
Let \(H\) be of the form~\eqref{eq:H-complex} satisfying~\eqref{eq:symmetry}
and~\eqref{eq:cond}.
Then there exist \(X_1\), \(X_2\in\mathbb{C}^{n\times n}\) and positive
numbers \(\lambda_1\) \(\lambda_2\), \(\dotsc\), \(\lambda_n\in\mathbb{R}\)
such that
\[
HX=X\begin{bmatrix} \Lambda_+ & \\ & -\Lambda_+ \end{bmatrix},
\qquad
Y\herm H=\begin{bmatrix} \Lambda_+ & \\ & -\Lambda_+ \end{bmatrix}Y\herm,
\qquad
Y\herm X=I_{2n},
\]
where
\[
X=\begin{bmatrix} X_1 & \conj X_2 \\ X_2 & \conj X_1 \end{bmatrix},
\qquad
Y=\begin{bmatrix} X_1 & -\conj X_2 \\ -X_2 & \conj X_1 \end{bmatrix},
\]
and \(\Lambda_+=\diag\set{\lambda_1,\dotsc,\lambda_n}\).
\end{theorem}
\begin{proof}
Let
\[
S=\begin{bmatrix} I_n & 0 \\ 0 & -I_n \end{bmatrix},
\qquad \Omega=\begin{bmatrix} A & B \\ \conj B & \conj A \end{bmatrix}.
\]
It follows from~\eqref{eq:GMG} that the eigenvalue problem \(Hx=\lambda x\) is
equivalent to \(Sx=\lambda^{-1}\Omega x\), which is a generalized
Hermitian--definite eigenvalue problem.
It is well known (see, for example,~\cite[Section 8.7]{GV1996}) that \(S\) and
\(\Omega\) are simultaneously diagonalizable by congruence transformations,%
\footnote{By congruence transformation, we mean a linear map on
\(\mathbb{C}^{m\times m}\) of the form \(X\mapsto C\herm XC\) where
\(C\in\mathbb{C}^{m\times m}\) is nonsingular.}
and hence \(H\) is diagonalizable and has real eigenvalues.
According to Theorem~\ref{thm:complex-symmetry}, there exist positive numbers
\(\lambda_1\) \(\lambda_2\), \(\dotsc\), \(\lambda_n\) such that \(H\) is
similar to \(\diag\set{\Lambda_+,-\Lambda_+}\), where
\(\Lambda_+=\diag\set{\lambda_1,\dotsc,\lambda_n}\).

To determine the structure of the eigenvectors of \(H\), we start with any
nonsingular matrix
\[
C=\begin{bmatrix} C_{11} & C_{12} \\ C_{21} & C_{22} \end{bmatrix}
\in\mathbb{C}^{2n\times 2n}
\]
satisfying \(C\herm \Omega C=I_{2n}\) and
\(C\herm SC=\diag\set{\Lambda_+,-\Lambda_+}^{-1}\).
Then \(C^{-1}HC=\diag\set{\Lambda_+,-\Lambda_+}\).
By setting \(X_1=C_{11}\Lambda_+^{1/2}\) and \(X_2=C_{21}\Lambda_+^{1/2}\), we
obtain that
\begin{equation}
\label{eq:eigvec_r+}
H\begin{bmatrix} X_1 \\ X_2 \end{bmatrix}
=\begin{bmatrix} X_1 \\ X_2 \end{bmatrix}\Lambda_+
\end{equation}
and
\begin{equation}
\label{eq:normalization}
X_1\herm X_1-X_2\herm X_2=I_n,
\end{equation}
It is straightforward to verify that the following equations are
equivalent to~\eqref{eq:eigvec_r+}:
\begin{align}
\label{eq:eigvec_r-}
&H\begin{bmatrix} \conj X_2 \\ \conj X_1 \end{bmatrix}
=-\begin{bmatrix} \conj X_2 \\ \conj X_1 \end{bmatrix}\Lambda_+,\\
\label{eq:eigvec_l+}
&\bigl[X_1\herm, -X_2\herm\bigr]H=\Lambda_+\bigl[X_1\herm, -X_2\herm\bigr],\\
\label{eq:eigvec_l-}
&\bigl[-\conj X_2\herm, \conj X_1\herm\bigr]H
=-\Lambda_+\bigl[-\conj X_2\herm, \conj X_1\herm\bigr].
\end{align}
Thus we have obtained all right and left eigenvectors of \(H\).
Finally, it follows from~\eqref{eq:eigvec_r+} and~\eqref{eq:eigvec_l-} that
\[
\bigl(\conj X_1\herm X_2-\conj X_2\herm X_1\bigr)\Lambda_+
=\bigl[-\conj X_2\herm, \conj X_1\herm\bigr]H
\begin{bmatrix} X_1 \\ X_2 \end{bmatrix}
=-\Lambda_+\bigl(\conj X_1\herm X_2-\conj X_2\herm X_1\bigr).
\]
Since \(\Lambda_+\) and \(-\Lambda_+\) have no common eigenvalue, the
homogeneous Sylvester equation \(Z\Lambda_+=-\Lambda_+Z\) has a unique
solution \(Z=0\).
Hence \(\conj X_1\herm X_2-\conj X_2\herm X_1=0\).
Combining this result with~\eqref{eq:normalization}, we conclude that
\(Y\herm X=I_{2n}\).
This completes the proof.
\end{proof}

We now consider a relatively simple case in which \(H\) is a real matrix.
In this case \eqref{eq:H-complex} simplifies to
\begin{equation}
\label{eq:H-real}
H=\begin{bmatrix} A & B \\ -B & -A \end{bmatrix},
\end{equation}
where both \(A\) and \(B\) are \(n\times n\) real symmetric matrices.
By definition \(H\) is a real Hamiltonian matrix.
Performing a symplectic orthogonal similarity transformation yields a block
cyclic form
\[
\frac{1}{\sqrt{2}}\begin{bmatrix} I_n & I_n \\ -I_n & I_n \end{bmatrix}\trans
\cdot H\cdot
\frac{1}{\sqrt{2}}\begin{bmatrix} I_n & I_n \\ -I_n & I_n \end{bmatrix}
=\begin{bmatrix} 0 & A+B \\ A-B & 0 \end{bmatrix}.
\]
This suggests that the eigenvalues of \(H\) are the square roots of the
eigenvalues of \((A+B)(A-B)\).
Since both \(A\) and \(B\) are real, \eqref{eq:cond} simplifies to that
\[
\begin{bmatrix} A & B \\ B & A \end{bmatrix}\GT0,
\]
or equivalently, \emph{both \(A+B\) and \(A-B\) are positive definite}.
Under this condition, the real Bethe--Salpeter eigenvalue problem is also
known as a \emph{linear response eigenvalue problem} which recently has attracted a
lot of attention (see, for example, \cite{BL2012,BL2013}).
In contrast to many recent developments~\cite{BL2013,RBLG2012} in linear
response eigenvalue problems that focus on large sparse eigensolvers, we
develop dense eigensolvers for this eigenvalue problem in the next section.

\subsection{Tamm--Dancoff approximation}
\label{subsec:TDA}

When the off-diagonal blocks of the Bethe--Salpeter Hamiltonian are set to
zero, known in the physics community as the Tamm--Dancoff
approximation~\cite{Dancoff1950,RL2000,Tamm1945}, the Bethe--Salpeter
eigenvalue problem reduces to a Hermitian eigenvalue problem.
One can use efficient algorithms available in ScaLAPACK to compute eigenpairs
of $A$.
In many cases, the results are found to be sufficiently close to the
eigenvalues of the full Bethe--Salpeter Hamiltonian and explain experiment.
However, in general, this simplification can lead to noticeable difference in
the computed spectrum.

In this subsection, we show that Tamm--Dancoff approximation consistently
overestimate the positive eigenvalues when the property~\eqref{eq:cond} holds.
More precisely, we have
\[
\lambda_j(H)\leq\lambda_j(A)
\]
for \(j=1\), \(2\), \(\dotsc\), \(n\), where \(\lambda_j(\cdot)\) denote the
\(j\)th \emph{largest} eigenvalue.
That is, \emph{every} positive eigenvalue obtained by TDA is greater than or
equal to the corresponding one of \(H\).
This theoretical result is consistent with several computational experiments
reported in~\cite{GMG2009,PMA2013}. However, we have not been able to 
find a rigorous proof of such a result in physics literature. To the best
of our knowledge, this result is not well known in the numerical linear 
algebra community, and its proof is not entirely trivial.

We provide a proof of this important property which we state clearly in
Theorem~\ref{thm:TDA}.  Our proof makes use of the following lemma which 
appeared relatively recently in \cite{BK2000}.

\begin{lemma}[\cite{BK2000}]
\label{lem:AM-GM}
Let \(A_1\), \(A_2\in\mathbb{C}^{n\times n}\) be Hermitian positive definite.
Then
\[
\sqrt{\lambda_j(A_1A_2)}\leq\lambda_j\Bigl(\frac{A_1+A_2}{2}\Bigr),
\]
for \(j=1\), \(2\), \(\dotsc\), \(n\).
\end{lemma}

With the help of this arithmetic--geometric inequality on eigenvalues, we
prove the claim we have made in the following theorem.

\begin{theorem}
\label{thm:TDA}
Let \(H\) be as defined in~\eqref{eq:H-complex}.
Then under the conditions~\eqref{eq:symmetry} and~\eqref{eq:cond}, we have
\[
\lambda_j(H)\leq\lambda_j(A)
\]
for \(j=1\), \(2\), \(\dotsc\), \(n\).
\end{theorem}
\begin{proof}
Let
\begin{equation}
\label{eq:embedding2}
\tilde A=\begin{bmatrix} \Re(A) & \Im(A) \\ -\Im(A) & \Re(A) \end{bmatrix},
\qquad
\tilde B=\begin{bmatrix} \Re(B) & -\Im(B) \\ -\Im(B) & -\Re(B) \end{bmatrix}.
\end{equation}
Then
\[
\frac{1}{\sqrt{2}}
\begin{bmatrix} I_n & -\mi I_n \\ I_n & \mi I_n \end{bmatrix}\herm
\cdot
\begin{bmatrix} A & B \\ \conj B & \conj A \end{bmatrix}
\cdot
\frac{1}{\sqrt{2}}
\begin{bmatrix} I_n & -\mi I_n \\ I_n & \mi I_n \end{bmatrix}
=\tilde A+\tilde B,
\]
indicating that~\eqref{eq:cond} is equivalent to \(\tilde A+\tilde B\GT0\).
Notice that
\[
\tilde A+\tilde B=J_n\trans(\tilde A-\tilde B)J_n.
\]
Therefore, \(\tilde A-\tilde B\) is also positive definite.
In the proof of Theorem~\ref{thm:equiv}, we have shown that \(H\) is unitarily
similar to \(-\mi J_n(\tilde A+\tilde B)\).
Since the eigenvalues of \(H\) are real and appear in pairs
\(\set{\lambda_j(H),-\lambda_j(H)}\), we obtain
\[
0<\lambda_j(H)=\lambda_j\bigl(-\mi J_n(\tilde A+\tilde B)\bigr)
=\sqrt{\lambda_{2j}\bigl([-\mi J_n(\tilde A+\tilde B)]^2\bigr)}
=\sqrt{\lambda_{2j}\bigl((\tilde A-\tilde B)(\tilde A+\tilde B)\bigr)}.
\]
Applying Lemma~\ref{lem:AM-GM} yields
\[
\sqrt{\lambda_{2j}\bigl((\tilde A-\tilde B)(\tilde A+\tilde B)\bigr)}
\leq\lambda_{2j}(\tilde A).
\]
Finally, since \(\Lambda(\tilde A)=\Lambda(A)\cup\Lambda(\conj A)
=\Lambda(A)\cup\Lambda(A)\), that is, the
eigenvalues of \(\tilde A\) are the same as those of \(A\) with doubled
multiplicity, we arrive at
\[
\lambda_{2j}(\tilde A)=\lambda_j(A).
\]
The theorem is thus proved.
\end{proof}

\section{Algorithms and implementations}
\label{sec:algorithms}
In this section, we present structure preserving algorithms for solving the
Bethe--Salpeter eigenvalue problem.
As we have shown in Theorem~\ref{thm:equiv}, Bethe--Salpeter eigenvalue
problems are equivalent to real Hamiltonian eigenvalue problems. Thus any
Hamiltonian eigensolver (see, for example, \cite{BKM2005}) can always be
used to solve this type of eigenvalue problem.
However, when~\eqref{eq:cond} holds, a more efficient algorithm can be
used to solve the Bethe--Salpeter eigenvalue problem.
Throughout this section the condition~\eqref{eq:cond} is assumed to be
satisfied.
In most cases, only the positive eigenvalues and the corresponding
eigenvectors are required for studying the properties of materials.
We will demonstrate how to compute them efficiently.

\subsection{Complex Bethe--Salpeter eigenvalue problems}
\label{subsec:complex}
From~\eqref{eq:GMG}, we can write
\[
Hx=\lambda x \qquad\Longleftrightarrow\qquad
\begin{bmatrix} I_n & 0 \\ 0 & -I_n \end{bmatrix}x
=\frac{1}{\lambda}\begin{bmatrix} A & B \\ \conj B & \conj A \end{bmatrix}x.
\]
A straightforward approach is to feed this problem to a generalized
Hermitian--definite eigensolver (for example, \texttt{ZHEGV} in
\lapack{}~\cite{LAPACK}).
However, this approach is in general \emph{not} structure preserving, meaning
that the computed eigenvalues and eigenvectors may not have the properties 
described in Theorem~\ref{thm:positive-definite}.
To see this, we analyze the algorithm implemented in \lapack{}'s
\texttt{ZHEGV}, which first computes the Cholesky factorization
\[
\begin{bmatrix} A & B \\ \conj B & \conj A \end{bmatrix}
=LL\herm
\]
and then applies a standard Hermitian eigensolver to the transformed problem
\[
L\begin{bmatrix} I_n & 0 \\ 0 & -I_n \end{bmatrix}L\herm.
\]
Let \(L+\Delta L\) be the computed Cholesky factor.
Then the backward error of \(H\) is
\[
\begin{bmatrix} \Delta_{11} & -\Delta_{21}\herm \\
\Delta_{21} & \Delta_{22} \end{bmatrix}
:=\begin{bmatrix} I_n & 0 \\ 0 & -I_n \end{bmatrix}
\biggl((L+\Delta L)(L+\Delta L)\herm
-\begin{bmatrix} A & B \\ \conj B & \conj A \end{bmatrix}\biggr),
\]
which does not necessarily satisfy \(\Delta_{21}\trans=\Delta_{21}\) and
\(\Delta_{22}=\conj\Delta_{11}\).
Therefore, the structure of \(H\) is destroyed in the sense that the error
introduced in the computed Cholesky factorization cannot be interpreted as a
structured backward error of \(H\).
Consequently, the properties given in Theorem~\ref{thm:positive-definite} are lost.
For instance, the eigenvalues of \((L+\Delta L)\diag\set{I_n,-I_n}(L+\Delta
L)\) do not necessarily appear in positive and negative pairs.

To develop a structure preserving approach, we make use of the observation we
made in Section~\ref{sec:properties}.
We have observed that \(Q\herm HQ=-\mi J_nM\), where
\[
Q=\frac{1}{\sqrt{2}}
\begin{bmatrix} I_n & -\mi I_n \\ I_n & \mi I_n \end{bmatrix},
\qquad
M=\begin{bmatrix} \Re(A+B) & \Im(A-B) \\ -\Im(A+B) & \Re(A-B) \end{bmatrix}
\GT0.
\]
Notice that both \(J_n\) and \(M\) are real matrices. Thus by working
with these matrices, we can avoid the use of complex arithmetic.
The basic steps of our algorithm are outlined in Algorithm~\ref{alg:complex},
in which we only compute the positive eigenvalues and the corresponding
eigenvectors.
Once the matrix \(M\) has been constructed, we perform a Cholesky
factorization, \(M=LL\trans\) to transform the non-Hermitian matrix, \(-\mi
J_nLL\trans\),
into the Hermitian matrix, \(-\mi L\trans J_nL\).
Let \(\Delta L\) be the error in the computed Cholesky factor \(L\).
Then
\[
\begin{bmatrix} \Delta_{11} & \Delta_{21}\trans \\
\Delta_{21} & \Delta_{22} \end{bmatrix}
:=(L+\Delta L)(L+\Delta L)\trans-M
\]
is the backward error of \(M\), and is real and symmetric.
Similar to the proof of Theorem~\ref{thm:equiv}, we set
\[
\Delta A=\frac{\Delta_{11}+\Delta_{22}}{2}
-\mi\frac{\Delta_{21}-\Delta_{21}\trans}{2},
\qquad \Delta B=\frac{\Delta_{11}-\Delta_{22}}{2}
-\mi\frac{\Delta_{21}+\Delta_{21}\trans}{2}.
\]
Then we have \((\Delta A)\herm=\Delta A\), \((\Delta B)\trans=\Delta B\), and
\[
\begin{bmatrix} \Delta_{11} & \Delta_{21}\trans \\
\Delta_{21} & \Delta_{22} \end{bmatrix}
=\begin{bmatrix} \Re(\Delta A+\Delta B) & \Im(\Delta A-\Delta B) \\
-\Im(\Delta A+\Delta B) & \Re(\Delta A-\Delta B) \end{bmatrix},
\]
indicating that the backward error in the computed Cholesky factorization of
\(M\) can be converted to a structured backward error of \(H\).
We remark that this analysis is valid even if \(\Delta L\) is \emph{not} lower
triangular.

In the next step we form \(L\trans J_nL\), which is a real skew-symmetric
matrix whose spectrum is symmetric with respect to the origin.
Applying a skew-symmetric eigensolver (for example, the one
in~\cite{WG1978b,WG1978}) to this matrix preserves the structure 
of \(\Lambda(H)\) and avoids the explicit use of complex arithmetic.
We will discuss the use of a skew-symmetric eigensolver in detail in
Section~\ref{subsec:parallel}.
Since any \(2n\times2n\) real skew-symmetric matrix is of the form \(C\trans
J_nC\), where \(C\in\mathbb{R}^{2n\times2n}\), the skew-symmetric
backward errors introduced in the construction of \(L\trans J_nL\) and in the
skew-symmetric eigensolver can be interpreted as a backward error (which does
not need to be lower triangular) of \(L\).
We have shown that the error in \(L\) can be converted to a structured
backward error of \(H\).
Thus this approach is structure preserving.
As a remark, we mention that there exists an alternative approach that
computes an SVD-like decomposition of \(L\) and avoids the explicit
construction of \(L\trans J_nL\), see~\cite{Xu2003,Xu2005} for details.

Once the eigenvalues and eigenvectors of \(L\trans J_nL\) have been computed,
the eigenvectors of \(H\) can be recovered by simply accumulating all
similarity transformations.
Complex arithmetic can also be avoided in this step by carefully manipulating
the real and imaginary parts of the eigenvectors.
As we have shown in Theorem~\ref{thm:positive-definite}, the left
eigenvectors, as well as the eigenvectors corresponding to negative
eigenvalues, can be restored with negligible effort according
to~\eqref{eq:eigvec_r-}--\eqref{eq:eigvec_l-} if needed.

\begin{algorithm}[!tb]
\caption{Algorithm for the complex Bethe--Salpeter eigenvalue problem}
\label{alg:complex}
\begin{algorithmic}[1]
\REQUIRE \(A=A\herm\), \(B=B\trans\in\mathbb{C}^{n\times n}\) such
that~\eqref{eq:cond} is satisfied.
\ENSURE \(X_1\), \(X_2\in\mathbb{C}^{n\times n}\) and
\(\Lambda_+=\diag\set{\lambda_1,\dotsc,\lambda_n}\) satisfying
\[
H\begin{bmatrix} X_1 \\ X_2 \end{bmatrix}
=\begin{bmatrix} X_1 \\ X_2 \end{bmatrix}\Lambda_+,
\qquad X_1\herm X_1-X_2\herm X_2=I_n,
\]
and \(\lambda_i>0\) for \(i=1\), \(\dotsc\), \(n\).
\STATE Construct
\[
M=\begin{bmatrix} \Re(A+B) & \Im(A-B) \\ -\Im(A+B) & \Re(A-B) \end{bmatrix}.
\]
\STATE Compute the Cholesky factorization \(M=LL\trans\).
\STATE Construct \(W=L\trans J_nL\), where \(J_n\) is defined in~\eqref{eq:J}.
\STATE Compute the spectral decomposition
\(-\mi W=[Z_+,Z_-]\diag\set{\Lambda_+,-\Lambda_+}[Z_+,Z_-]\herm\).
\label{algstep:skew-sym}
\STATE Set
\[
\begin{bmatrix} X_1 \\ X_2 \end{bmatrix}
=\begin{bmatrix} I_n & 0 \\ 0 & -I_n \end{bmatrix}QLZ_+\Lambda_+^{-1/2}.
\]
\label{algstep:eigenvector}
\end{algorithmic}
\end{algorithm}

\subsection{Real Bethe--Salpeter eigenvalue problems}
\label{subsec:real}
As we have seen in Section~\ref{sec:properties}, the real Bethe--Salpeter
eigenvalue problem can be reduced to a product eigenvalue problem.
Directly feeding \(A+B\) and \(A-B\) to a symmetric--definite eigensolver
squares the eigenvalues and can potentially spoil the accuracy of the
eigenvalues and eigenvectors.
When the accuracy requirement is high, an alternative approach is needed.
Suppose that \(A+B=L_1L_1\trans\), \(A-B=L_2L_2\trans\).
Then \((A+B)(A-B)\) is similar to \((L_2\trans L_1)(L_2\trans L_1)\trans\).
Notice that the spectral decomposition of
\((L_2\trans L_1)(L_2\trans L_1)\trans\) can be obtained from the singular
value decomposition of \(L_2\trans L_1\).
Theorem~\ref{thm:real-BSE} summarizes this observation.

\begin{theorem}
\label{thm:real-BSE}
Let \(H\) be defined by~\eqref{eq:H-real} satisfying that \(A+B\GT0\),
\(A-B\GT0\).
Suppose that \(L_2\trans L_1=U\Lambda_+ V\trans\) is the singular value
decomposition of \(L_2\trans L_1\), where \(L_1\),
\(L_2\in\mathbb{R}^{n\times n}\) satisfy that \(L_1L_1\trans=A+B\)
and \(L_2L_2\trans=A-B\).
Then the spectral decomposition of \(H\) is given by
\[
H=X\begin{bmatrix} \Lambda_+ & 0 \\ 0 & -\Lambda_+ \end{bmatrix}Y\trans,
\]
where
\[
X=\frac{1}{2}\begin{bmatrix}
L_2U+L_1V & L_2U-L_1V \\
L_2U-L_1V & L_2U+L_1V
\end{bmatrix}
\begin{bmatrix} \Lambda_+^{-1/2} & 0 \\ 0 & \Lambda_+^{-1/2} \end{bmatrix}
\]
and
\[
Y=X\itrans=\frac{1}{2}\begin{bmatrix}
L_1V+L_2U & L_1V-L_2U \\
L_1V-L_2U & L_1V+L_2U
\end{bmatrix}
\begin{bmatrix} \Lambda_+^{-1/2} & 0 \\ 0 & \Lambda_+^{-1/2} \end{bmatrix}.
\]
\end{theorem}
\begin{proof}
It suffices to verify that \(HX=X\diag\set{\Lambda_+,-\Lambda_+}\) and
\(Y\trans X=I_{2n}\), which only require simple algebraic manipulations.
\end{proof}

Based on this theorem, we propose Algorithm~\ref{alg:real} as a real
Bethe--Salpeter eigensolver.
We remark that such an eigensolver can also be used as a dense kernel within
structure preserving projection methods for linear response eigenvalue
problems~\cite{BL2013,KMS2014}.
Just like the complex case, the negative eigenvalues and the corresponding
eigenvectors, if needed, can be easily constructed from the outputs of
Algorithm~\ref{alg:real} according to Theorems~\ref{thm:positive-definite}
and~\ref{thm:real-BSE}.

\begin{algorithm}[!tb]
\caption{Algorithm for the real Bethe--Salpeter eigenvalue problem}
\label{alg:real}
\begin{algorithmic}[1]
\REQUIRE \(A=A\trans\), \(B=B\trans\in\mathbb{R}^{n\times n}\) such that
\(A+B\GT0\) and \(A-B\GT0\).
\ENSURE \(X_1\), \(X_2\in\mathbb{R}^{n\times n}\) and
\(\Lambda_+=\diag\set{\lambda_1,\dotsc,\lambda_n}\) satisfying
\[
H\begin{bmatrix} X_1 \\ X_2 \end{bmatrix}
=\begin{bmatrix} X_1 \\ X_2 \end{bmatrix}\Lambda_+,
\qquad X_1\trans X_1-X_2\trans X_2=I_n,
\]
and \(\lambda_i>0\) for \(i=1\), \(\dotsc\), \(n\).
\STATE Compute the Cholesky factorizations \(A+B=L_1L_1\trans\),
\(A-B=L_2L_2\trans\).
\STATE Compute the singular value decomposition
\(L_2\trans L_1=U\Lambda_+V\trans\).
\STATE Set \(X_1=L_2U+L_1V\), \(X_2=L_2U-L_1V\).
\end{algorithmic}
\end{algorithm}

\subsection{Parallel implementations}
\label{subsec:parallel}
To solve large scale Bethe--Salpeter eigenvalue problems arising from quantum
physics, parallelization of Algorithms~\ref{alg:complex} and~\ref{alg:real} on
distributed memory systems is required in practical computations.
All \(O(n^3)\) operations in Algorithm~\ref{alg:real} consist of basic
linear algebra operations, and can be accomplished by calling linear algebra
libraries such as \blas{}/\lapack{}.
There also exist \scalapack{} subroutines for these operations, which allow us
to parallelize the algorithm in a straightforward manner.
So we will mainly discuss implementation issues for
Algorithm~\ref{alg:complex}.

The main obstacle to efficiently implementing Algorithm~\ref{alg:complex} is
the lack of a skew-symmetric eigensolver in \lapack{}/\scalapack{}.
The algorithm described in~\cite{WG1978b,WG1978} is based on level 1 BLAS
operations, hence is not efficient on modern architectures with memory
hierarchy.
Therefore, we have to implement step~\ref{algstep:skew-sym} in
Algorithm~\ref{alg:complex} by ourselves.
To make use of \scalapack{} as much as possible, we propose the following
strategy shown in Algorithm~\ref{alg:detail}.
Several remarks on the implementation are in order:
\begin{enumerate}
\item Tridiagonal reduction of a skew-symmetric matrix can be accomplished by
applying a sequence of Householder reflections.
This is in fact slightly simpler compared to reducing a symmetric matrix to
tridiagonal form.
We implemented a modified version of \scalapack{}'s \texttt{PDSYTRD} to achieve
this goal.
Several \blas{}/\pblas{}-like subroutines for skew-symmetric matrix
operations are also implemented.
\item Suppose that
\[
T=\tridiag
\begin{Bmatrix}
& \alpha_1 & & \cdots & & \alpha_{2n-1} & \\
0 & & \cdots & & \cdots & & 0 \\
& -\alpha_1 & & \cdots & & -\alpha_{2n-1} & \\
\end{Bmatrix}.
\]
Then
\[
-\mi D\herm TD=\tridiag
\begin{Bmatrix}
& \alpha_1 & & \cdots & & \alpha_{2n-1} & \\
0 & & \cdots & & \cdots & & 0 \\
& \alpha_1 & & \cdots & & \alpha_{2n-1} & \\
\end{Bmatrix}
\]
is a real symmetric tridiagonal matrix, whose spectral decomposition can be
easily computed by calling \scalapack{}.
This technique is essentially the same as the one in~\cite{WG1978}.
We then use the bisection method (\texttt{PDSTEBZ}/\texttt{PDSTEIN}) to
compute the positive eigenvalues and the corresponding eigenvectors.
The eigenvectors are reorthogonalized when the accuracy requirement is high.
\item Step~\ref{algstep:complex-matmult} in Algorithm~\ref{alg:detail} is
\emph{not} performed by \pblas{}.
Because \(Q\) and \(D\) are both known, the application of these unitary
transformations can be accomplished with \(O(n^2)\) operations.
\item In Step~\ref{algstep:complex-matmult2} of Algorithm~\ref{alg:detail}, we
separate the computation for real and imaginary parts:
\(\Re(X_+)=\Re(\hat X)V_+\), \(\Im(X_+)=\Im(\hat X)V_+\).
This is based on the fact that \(V_+\) is real, and one \texttt{PZGEMM} call
is about twice as expensive as two \texttt{PDGEMM} calls.
\end{enumerate}

\begin{algorithm}[!tb]
\caption{Parallel implementation of steps~\ref{algstep:skew-sym}
and~\ref{algstep:eigenvector} in Algorithm~\ref{alg:complex}}
\label{alg:detail}
\begin{algorithmic}[1]
\STATE Tridiagonal reduction: \(W=UTU\trans\).
\STATE Compute the positive spectral decomposition
\(-\mi D\herm TD=[V_+,V_-]\diag\set{\Lambda_+,-\Lambda_+}[V_+,V_-]\trans\),
where \(D=\diag\set{1,\mi,\mi^2,\dotsc,\mi^{2n}}\).
\STATE Construct \(\Phi=LU\) by applying a sequence of Householder reflections.
\STATE Construct \(\hat Y=Q\Phi D\).
\label{algstep:complex-matmult}
\STATE Construct \(Y_+=\hat YV_+\) using \pblas{}.
\label{algstep:complex-matmult2}
\STATE Set
\[
\begin{bmatrix} X_1 \\ X_2 \end{bmatrix}
=\begin{bmatrix} I_n & 0 \\ 0 & -I_n \end{bmatrix}Y_+.
\]
\end{algorithmic}
\end{algorithm}

Finally, we remark that our parallel algorithms/implementations are just
proof-of-concept.
There is certainly room for improvement.
For example, our tridiagonal reduction step is modified from \scalapack{}'s
\texttt{PDSYTRD}, which is relatively simple, but is not state-of-the-art.
Many modern implementations of symmetric tridiagonal reduction are based on
successive band reduction~\cite{BLS2000,LLD2011,ELPA}.
The successive band reduction technique extends naturally to skew-symmetric
matrices.
There are also many alternative tridiagonal eigensolvers other than the
bisection method, for example, the MRRR method~\cite{DPV2006,WL2013}.
But improvements in these directions exceed the scope of this paper.

\section{Numerical experiments}
\label{sec:experiments}
In this section, we present numerical examples for three matrices obtained
from discretized Bethe--Salpeter equations.
The numerical experiments are performed on the Cray XE6 machine, Hopper, at
the National Energy Research Scientific Computing Center (NERSC).%
\footnote{\url{https://www.nersc.gov/users/computational-systems/hopper/}}
Each Hopper node consists of two twelve-core AMD `MagnyCours' 2.1-GHz processors,
and has 32 GB DDR3 1333-MHz memory. Each core has its own L1 and L2 caches, with 64 KB and 512 KB, respectively. Hopper's compute nodes are connected via 
a 3D torus Cray Gemini network with a maximum bandwidth of 8.3 gigabytes per 
second. The internode latency ranges from 1.27 microseconds to 1.38 microseconds. The latency between cores is less than 1 microsecond.

The examples we use here correspond to discretized Bethe--Salpeter Hamiltonians
associated with naphthalene, gallium arsenide (GaAs), and boron nitride (BN),
respectively.
The dimensions of these \(2n\times2n\) Hamiltonians are \(64\times64\),
\(256\times256\), and \(4608\times4608\), respectively.
We implemented Algorithm~\ref{alg:complex} in Fortran 90, using the message
passing interface (MPI) for parallelization.
We used \scalapack{} to perform some basic parallel matrix computations.
No multithreading features such as OpenMP or threaded linear algebra libraries
are used.
To make fair comparisons, all eigenvalues and eigenvectors are calculated.

We first compare our implementation of Algorithm~\ref{alg:complex} with
\lapack{}'s non-Hermitian eigensolver \texttt{ZGEEV}.
Test runs are performed using a single core so that both solvers are
sequential.
From Table~\ref{tab:residual} we see that both solvers produce solutions
with small residuals, and Algorithm~\ref{alg:complex} is in general more
accurate.
We also compute the eigenvalues of \(A\) using \lapack{}'s Hermitian
eigensolver \texttt{ZHEEV}.
In Figure~\ref{fig:DOS}, we plot the spectral density of the computed
eigenvalues with and without TDA.
The delta function in~\eqref{eq:DOS} is approximated using the Gaussian
function, that is,
\[
\delta(t)\approx\frac{1}{\sqrt{2\pi}\,\sigma}
\exp\Bigl(-\frac{t^2}{2\sigma^2}\Bigr),
\]
with standard deviation \(\sigma=5\times10^{-4}\).
The figure illustrates the difference between \(\Lambda(A)\) and
\(\Lambda(H)\) for the naphthalene example.
We observe up to \(12\%\) relative differences in eigenvalues in this example
when TDA is used.
We also see from the figure that \(\Lambda(A)\) is always to the right of
\(\Lambda(H)\).
For the other two examples, the error introduced by TDA are up to \(1.2\%\)
and \(0.93\%\), respectively.
Computational results confirm that all eigenvalues obtained by TDA are larger
than the true eigenvalues as predicted by Theorem~\ref{thm:TDA}.
We remark that \texttt{ZGEEV} produces quite accurate eigenvalues in these
examples, despite the fact that all computed eigenvalues are not real.

\begin{table}[!tb]
\centering
\caption{Comparison between \texttt{ZGEEV} and Algorithm~\ref{alg:complex}.}
\label{tab:residual}
\begin{tabular}{lcccc}
\hline
& & naphthalene & ~~~~GaAs~~~~ & ~~~~~BN~~~~~ \\
& & $n=32$ & $n=128$ & $n=2304$ \\
\hline
\texttt{ZGEEV}
& $\|Y\herm HX-\Lambda\|_F/\|H\|_F$
& $3.8\times10^{-15}$ & $8.5\times10^{-15}$ & $2.0\times10^{-14}$ \\
& $\|Y\herm X-I_{2n}\|_F/\sqrt{2n}$
& $2.8\times10^{-15}$ & $9.0\times10^{-15}$ & $2.1\times10^{-14}$ \\
\hline
Algorithm~\ref{alg:complex}
& $\|Y\herm HX-\Lambda\|_F/\|H\|_F$
& $1.5\times10^{-15}$ & $3.3\times10^{-15}$ & $5.4\times10^{-15}$ \\
& $\|Y\herm X-I_{2n}\|_F/\sqrt{2n}$
& $1.1\times10^{-15}$ & $3.1\times10^{-15}$ & $4.3\times10^{-15}$ \\
\hline
\end{tabular}
\end{table}

\begin{figure}[!tb]
\centering
\includegraphics[width=0.6\textwidth]{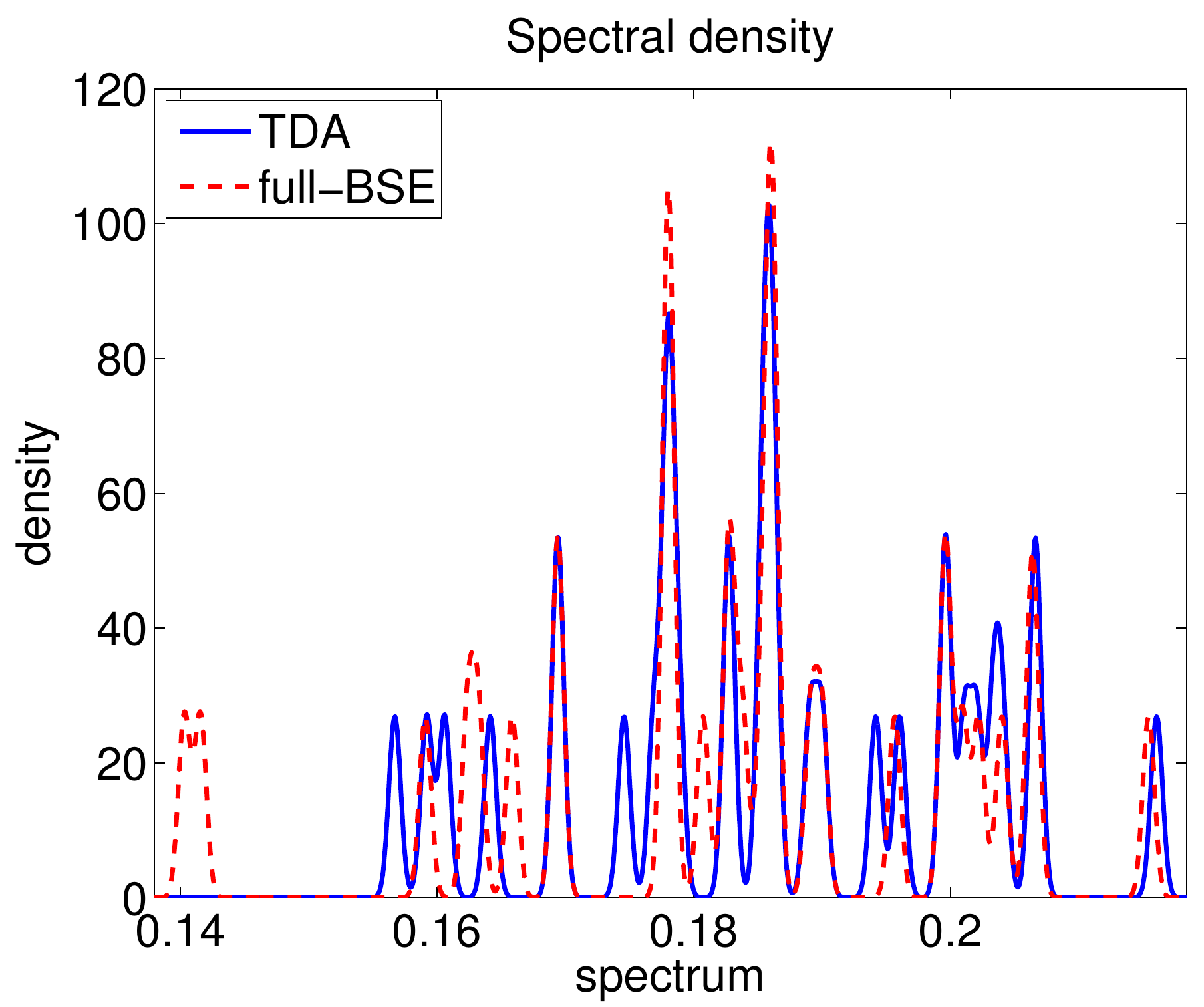}
\caption{Comparison of the spectral density, defined by
Equation~\eqref{eq:DOS}, for the full Bethe--Salpeter eigenvalue problem and
the BSE constructed without the off-diagonal blocks, with the Tamm--Dancoff
approximation.}
\label{fig:DOS}
\end{figure}

Table~\ref{tab:time} contains the execution time of different approaches.
Algorithm~\ref{alg:complex} is about five times faster than the non-Hermitian
solver \texttt{ZGEEV}.
An interesting observation is that TDA is \emph{not} much faster than our
full-BSE solver, especially when \(n\) gets large.
This is not a surprising result.
In fact, the major cost of Algorithm~\ref{alg:complex} is diagonalizing a real
\(2n\times2n\) skew-symmetric matrix, which is comparable to the cost of
diagonalizing a complex \(n\times n\) Hermitian matrix.

\begin{table}[!tb]
\centering
\caption{Execution time of different approaches.}
\label{tab:time}
\begin{tabular}{lccc}
\hline
& naphthalene & ~~~~GaAs~~~~ & ~~~~~BN~~~~~ \\
& $n=32$ & $n=128$ & $n=2304$ \\
\hline
\texttt{ZGEEV}              & $1.8\times10^{-2}$ & $4.8\times10^{-1}$ & $1.2\times10^{3}$ \\
\hline
Algorithm~\ref{alg:complex} & $4.1\times10^{-3}$ & $7.6\times10^{-2}$ & $1.6\times10^{2}$ \\
\hline
\texttt{ZHEEV} (TDA)        & $1.6\times10^{-3}$ & $5.2\times10^{-2}$ & $2.5\times10^{2}$ \\
\hline
\end{tabular}
\end{table}

Finally, we perform a simple scalability test on our parallel implementation
of Algorithm~\ref{alg:complex}.
We use the BN example since it is of moderate size.
Test runs are performed on \(1\times1\), \(2\times2\), \(\dotsc\),
\(9\times9\), \(10\times10\) processor grids, with block factor \(n_b=64\) for
the 2D block cyclic data layout.
In Figure~\ref{fig:scale}, we illustrate the overall execution time, as well
as the performance profile for each component of the algorithm.
We can see from the figure that all components scale similarly in this
example.
This indicates that Algorithm~\ref{alg:complex} scales reasonably well as the
overall parallel scalability is close to that of the \texttt{PDGEMM}
component.

\begin{figure}[!tb]
\centering
\includegraphics[width=0.7\textwidth]{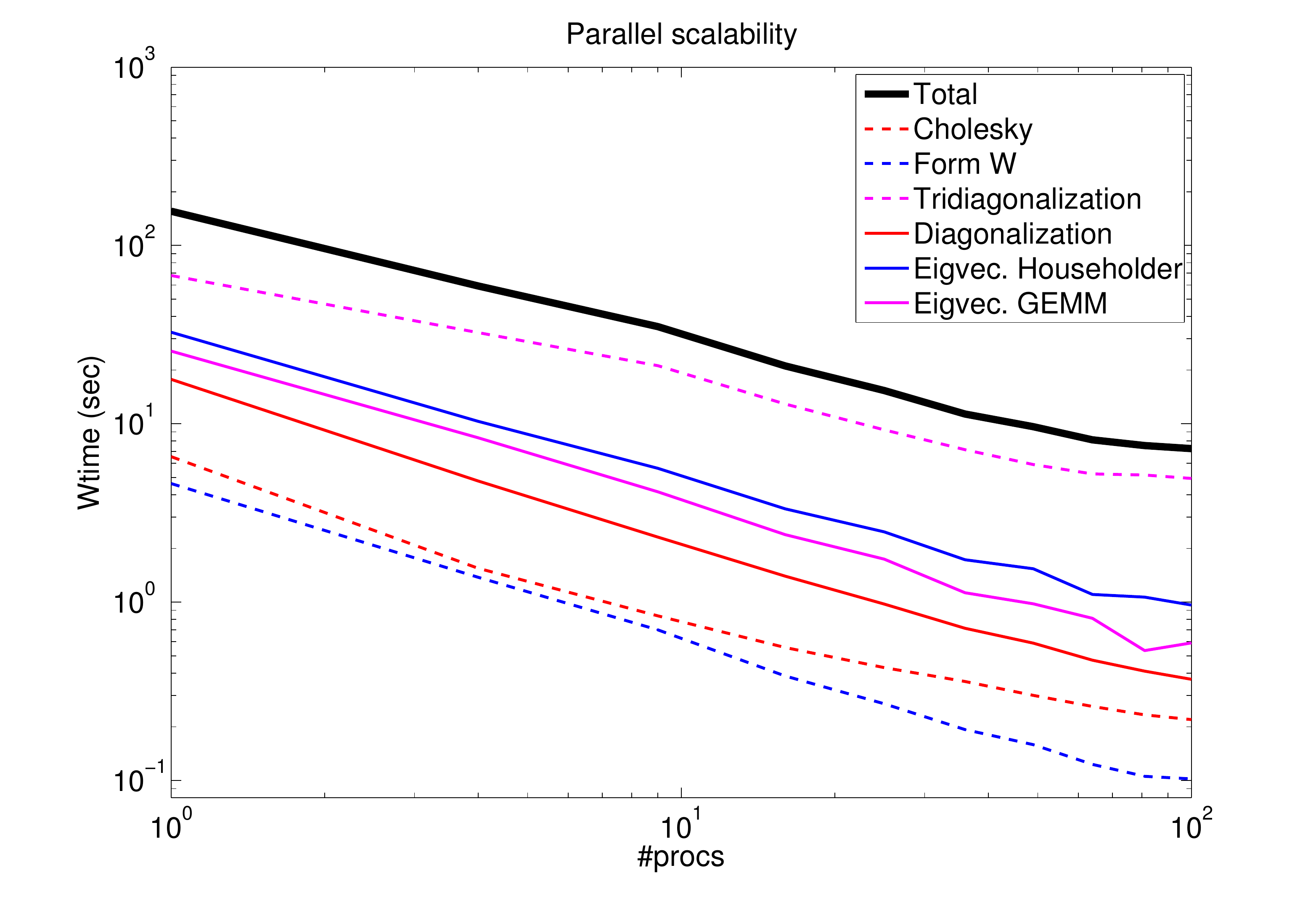}
\caption{Parallel scalability of Algorithm~\ref{alg:complex} for BN
(\(H\in\mathbb{C}^{4608\times4608}\)).}
\label{fig:scale}
\end{figure}

\section{Conclusions}
\label{sec:conclusions}
We showed, in this paper, that the Bethe--Salpeter eigenvalue problem is
equivalent to a real Hamiltonian eigenvalue problem and can always be solved
by an efficient Hamiltonian eigensolver.
For complex Bethe--Salpeter Hamiltonians that satisfy the additional
property~\eqref{eq:cond}, which almost always holds in practice,  it is
possible to compute all of its eigenpairs by a structure preserving algorithm
we developed in this paper.
When the Hamiltonian is real, we can turn the eigenvalue problem into a
product eigenvalue problem. We presented an efficient and reliable way to
solve this eigenvalue problem.

When the Tamm--Dancoff approximation is used, the Bethe--Salpeter eigenvalue
problem reduces to a Hermitian eigenvalue problem that can be solved by
existing tools in the \scalapack{} library.
We showed that Tamm--Dancoff approximation always overestimate all
eigenvalues.

We presented numerical algorithms that demonstrated the accuracy and
efficiency of our algorithm.
However, we should point out that our parallel implementation is preliminary,
and there is plenty of room for improvements.

\section*{Acknowledgments}
The authors thank Zhaojun Bai, Peter Benner, Fabien Bruneval, Heike
Fa{\ss}bender, Daniel Kressner, and Hongguo Xu for fruitful discussions.
The authors are also grateful to anonymous referees for careful reading and
providing valuable comments.
This material is based upon work supported by the Scientific Discovery through
Advanced Computing (SciDAC) Program on Excited State Phenomena in Energy
Materials funded by the U.S. Department of Energy, Office of Science, under
SciDAC program, the Offices of Advanced Scientific Computing Research and
Basic Energy Sciences contract number DE-AC02-05CH11231 at the Lawrence
Berkeley National Laboratory.
This research used resources of the National Energy Research Scientific
Computing Center, a DOE Office of Science User Facility supported by the
Office of Science of the U.S. Department of Energy under Contract No.\
DE-AC02-05CH11231.

\end{document}